\documentclass[12pt]{article}

\usepackage{comment}
\usepackage{pgf,tikz}
\usepackage{pgfplots}
\usetikzlibrary{arrows.meta}

\usepackage{amsmath,amssymb,amsthm,cite}
\usepackage[colorlinks,citecolor=red,urlcolor=blue,linkcolor=blue]{hyperref}
\usepackage{mathrsfs,dsfont}
\usepackage{graphicx}
\usepackage{enumitem}

\newtheorem{nummer}{ }
\newtheorem{thm}[nummer]{\bf Theorem}

\newtheorem{lem}[nummer]{\bf Lemma}
\newtheorem{cor}[nummer]{\bf Corollary}
\newtheorem{fct}[nummer]{\bf Fact}

\newcommand\nT{\mathscr{T}}

\newcommand{\ie} {\sl i.e.}

\newcommand{\Q}{\mathds{Q}}

\newcommand{\Z}{\mathds{Z}}

\newcommand{\F}{\mathds{F}}

\newcommand{\THM}{Theorem}
\newcommand{\FCT}{Fact}
\newcommand{\LEM}{Lemma}

\newcommand{\nO}{\mathscr{O}}

\newcommand{\ta}{\tilde a}
\newcommand{\tb}{\tilde b}

\newcommand{\dritt}[2]{#1\mathbin{\#}#2}

\newcommand{\C}{\Gamma}
\newcommand{\Cabc}{\C_{\alpha,\beta,\gamma}}
\newcommand{\Cabcd}{\C_{\alpha,\beta,\gamma,\delta}}
\newcommand{\Cab}{\C_{a,b}}
\newcommand{\Cuv}{\C_{u,v}}
\newcommand{\Ct}{\C_{t}}

\newcommand{\konj}[1]{\bar{#1}}

\makeatletter
\def\opargproof[#1]{\par\noindent {\bf #1 }}
 \makeatother

\definecolor{darkgreen}{rgb}{0,.6,0}

\begin{document}
\begin{center}
{\Large\bf A Geometric Approach to Elliptic Curves with}\\[1.8ex] 
{\Large\bf Torsion Groups $\mathbf{\Z/10\Z}$, 
$\mathbf{\Z/12\Z}$, $\mathbf{\Z/14\Z}$, and $\mathbf{\Z/16\Z}$}

\medskip
{\small Lorenz Halbeisen}\\[1.2ex] 
{\scriptsize Department of Mathematics, ETH Zentrum,
R\"amistrasse\;101, 8092 Z\"urich, Switzerland\\ lorenz.halbeisen@math.ethz.ch}\\[1.8ex]
{\small Norbert Hungerb\"uhler}\\[1.2ex] 
{\scriptsize Department of Mathematics, ETH Zentrum,
R\"amistrasse\;101, 8092 Z\"urich, Switzerland\\ norbert.hungerbuehler@math.ethz.ch}\\[1.8ex]
{\small Maksym Voznyy}\\[1.2ex] 
{\scriptsize Department of Technology, Stephen Leacock CI,
Toronto District School Board, Toronto, Canada\\ maksym.voznyy@tdsb.on.ca}\\[1.8ex]
{\small Arman Shamsi Zargar}\\[1.2ex] 
{\scriptsize Department of Mathematics and Applications,
University of Mohaghegh Ardabili, Ardabil, Iran\\ zargar@uma.ac.ir}

\end{center}

\hspace{5ex}{\small{\it key-words\/}: elliptic curve, parametrisation, rank, quadratic field, torsion group}

\hspace{5ex}{\small{\it 2020 Mathematics Subject 
Classification\/}: {\bf 11G05}, 14H52

\begin{abstract}\noindent
{\small We give new parametrisations of elliptic curves in 
Weierstrass normal form $y^2=x^3+ax^2+bx$ with torsion 
groups $\Z/10\Z$ and $\Z/12\Z$ over $\Q$, 
and with $\Z/14\Z$ and $\Z/16\Z$ over quadratic fields.
Even though the parametrisations are equivalent to those given by 
Kubert and Rabarison, respectively, 
with the new parametrisations 
we found three infinite families of elliptic 
curves with torsion group $\Z/12\Z$ and positive rank. Furthermore,
we found elliptic curves with torsion group $\Z/14\Z$ 
and rank~$3$\,--\,which is a new record for such curves,
as well as some new elliptic 
curves with torsion group $\Z/16\Z$ 
and rank~$3$.}
\end{abstract}

\section{Introduction}

An elliptic curve $E$ over a field $K$ is a smooth projective 
curve of genus~$1$ equipped 
with a $K$-rational point. When embedded in the affine plane, 
$E$ is described by the Weierstrass model
$y^2+a_1xy+a_3y=x^3+a_2x^2+a_4x+a_6$, where the coefficients belong to~$K$. 
Elliptic curves can be represented by several other equations. 
The interested reader may consult \cite[Ch.\;2]{Washington}. 
In the few past decades, many alternative equations describing 
$E$ have been introduced in the context of cryptographic applications.

Given an elliptic curve $E$ defined over a field~$K$, the 
Mordell-Weil theorem shows that as an abelian group $E(K)$ 
is finitely generated over a number field.
In particular, $E(K) \cong \mathcal{T} \times \mathbb{Z}^r$, 
where the torsion group $\mathcal{T}$ is finite. 
The non-negative integer $r$ is called the rank. 

Many number theorists have tried to construct families of 
elliptic curves with rank as high as possible. The rank 
of an elliptic curve measures, in some sense, the 
number of rational points on the curve.  
Despite this, there is no known algorithm guaranteed to compute the rank. 

With a geometric approach developed in~\cite{HHSchroeter}, 
we investigate the rank of elliptic curves
with torsion groups $\Z/10\Z$ and $\Z/12\Z$ over $\Q$, 
and with torsion groups $\Z/14\Z$ and $\Z/16\Z$ over quadratic fields.
In particular, we give new parametrisations of elliptic curves in 
Weierstrass normal form for these curves.
Even though the parametrisations are equivalent to those given by 
Kubert\;\cite{Kubert} and Rabarison\;\cite{Rabarison}, respectively, 
especially the parametrisation of elliptic curves with 
torsion group $\Z/14\Z$ and $\Z/16\Z$ over quadratic fields 
are novel (see \cite{DujellaQuad}).
By our approach, we are able to find parametrisations of 
elliptic curves with torsion groups 
$\Z/10\Z$ and $\Z/12\Z$, 
and we provide three infinite families of 
curves with torsion group $\Z/12\Z$ and positive rank. Furthermore,
we find elliptic curves with torsion group $\Z/14\Z$ 
and rank~$3$\,--\,which is a new record for such curves,
as well as some new elliptic 
curves with torsion group $\Z/16\Z$ and rank~$3$.
Consult~\cite{dujella-web,DujellaQuad} for the current records 
on the rank of elliptic curves (with prescribed torsion groups)
over the rational and quadratic fields.

\section{A Geometric Approach to Elliptic Curves}

In this section we present a geometric approach to elliptic curves
with torsion groups $\Z/2n\Z$ over arbitrary fields. The approach 
is bases on a ruler construction of cubic curves due to 
Schroeter~\cite{Schroeter}, which was further developed and 
applied to elliptic curves in~\cite{HHSchroeter}.

Let $\F$ be a finite field extension of $\Q$ and 
let $$\Cab:\;y^2=x^3+ax^2+bx$$ 
with $a,b\in\F$ be a cubic curve. Furthermore, let $T:=(0,0)$. Then $T$
belongs to~$\Cab$ and $T+T=\nO$, where $\nO$ denotes the neutral element
of the Mordell-Weil group of $\Cab$ and ``$+$'' is the group operation. 
If $A$ is a point on~$\Cab$, then we call
the point $\konj{A}:=T+A$ the \hbox{\it conjugate of\/ $A$.}
Since $T+T=\nO$, we have $$\konj{\konj{A}}=T+\konj{A}=T+T+A=\nO+A=A\,.$$
Furthermore, for points $A,B$ on an elliptic curve $\C$ (over $\F$), 
let $$\dritt AB:=-(A+B)\,.$$
In particular, if $C=\dritt AA$, then the line through $C$ 
and $A$ is tangent to $\Cab$ with contact point~$A$.

The following fact gives a connection between conjugate points and
tangents (see Figure~\ref{bild:fact1}).

\begin{fct}\label{fct:tangent}
If\/ $A,\konj{A},B$ are three points on\/ $\Cab$ which lie on a straight
line, then\/ $\dritt AA=\konj{B}$.
\end{fct}
\begin{proof} If $A,\konj{A},B$ are three points on $\Cab$ on a straight
line, then $A+\konj{A}=-B$. Thus, $A+T+A=T+A+A=-B$, which implies
$$A+A=T+(T+A+A)=T+(-B)=(-T)+(-B)=-(T+B)=-\konj{B}\,,$$ and therefore,
the line $A\konj{B}$ is tangent to $\Cab$ with contact point~$A$, {\ie},
$\dritt AA=\konj{B}$.
\end{proof}
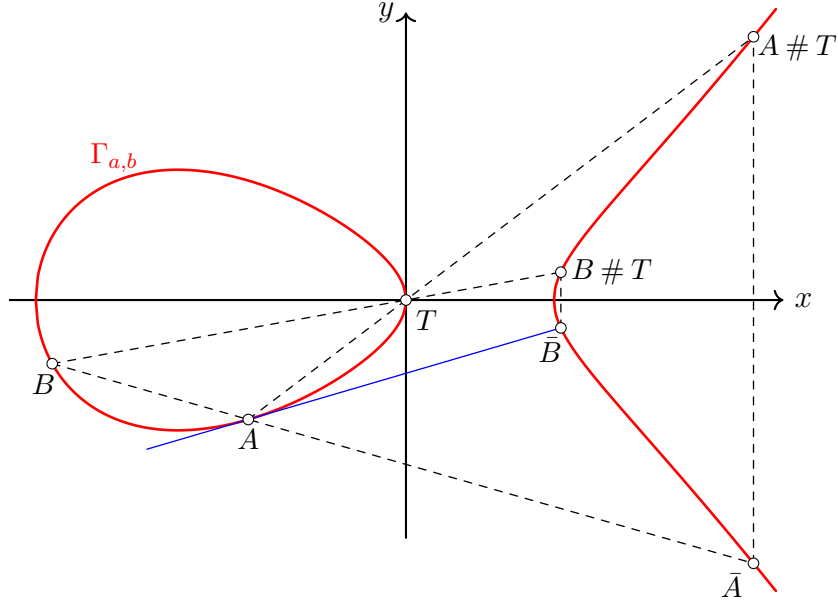
\begin{figure}[h!]
\begin{center}
\begin{tikzpicture}[line cap=round,x=28,y=9]
\clip(-5.35,-13.5) rectangle (6.3,13.);
\draw[-{>[length=1.3mm]},line width=.8pt] (-10,0)--(5.1,0) node[right] {$x$};
\draw[-{>[length=1.3mm]},line width=.8pt] (0,-10)--(0,12.1)node[left] {$y$};

\draw[line width=1.pt,color=red,smooth,samples=200,domain=-5:-0] plot( \x, {sqrt( (\x)^3 + 3*(\x)^2 - 10*\x) } );
\draw[line width=1.pt,color=red,smooth,samples=200,domain=-5:-0] plot( \x, {-sqrt( (\x)^3 + 3*(\x)^2 - 10*\x)} );
\draw[line width=1.pt,color=red,smooth,samples=200,domain=2:5] plot( \x, {sqrt( (\x)^3 + 3*(\x)^2 - 10*\x) } );
\draw[line width=1.pt,color=red,smooth,samples=200,domain=2:5] plot( \x, {-sqrt( (\x)^3 + 3*(\x)^2 - 10*\x) } );

\draw[line width=.5pt,dashed,domain=-4.7789:4.69484] plot( \x, { -6.91296 - 0.88653 *\x} );
\draw[line width=.5pt,domain=-3.5:2.09253,blue] plot( \x, {-3.08117 + 0.912432*\x} );
\draw[line width=.5pt,dashed,domain=-2.13:4.69484] plot( \x, {2.35899*\x} );
\draw[line width=.5pt,dashed,domain=-4.7789:2.09253] plot( \x, {0.560028*\x} );
\draw[line width=.5pt,dashed] (4.69484,-11.0751)--(4.69484,11.0751);
\draw[line width=.5pt,dashed] (2.09253, -1.17188)--(2.09253, 1.17188);

\begin{small}
\draw [fill=white] (0,0) circle (2.pt);
\draw [fill=white] (-2.13,-5.02465) circle (2.pt);
\draw [fill=white] (4.69484, -11.0751) circle (2.pt);
\draw [fill=white] (4.69484, 11.0751) circle (2.pt);
\draw [fill=white] (-4.7789, -2.67632) circle (2.pt);
\draw [fill=white] (2.09253, -1.17188) circle (2.pt);
\draw [fill=white] (2.09253, 1.17188) circle (2.pt);

\draw[color=black] (0,0) node[anchor=north west] {$T$};
\draw[color=black] (-2.13,-5.02465) node[anchor=north ] {$A$};
\draw[color=black] (4.4, -11.0751) node[anchor= north] {$\bar A$};
\draw[color=black] (4.61, 11.6) node[anchor= north west] {$ \dritt AT$};
\draw[color=black] (-4.6, -2.67632) node[anchor=north east] {$B$};
\draw[color=black] (1.95, -1.17188) node[anchor=north ] {$\bar B$};
\draw[color=black] (2.09253, 1.17188) node[anchor=west ] {$\dritt BT$};
\draw[color=red] (-4.4,6) node[anchor=west ] {$\Gamma_{a,b}$};

\end{small}
\end{tikzpicture}
\caption{Conjugate points and tangents.}\label{bild:fact1}
\end{center}
\end{figure}
In homogeneous coordinates, the curve $y^2=x^3+ax^2+bx$ becomes
$$\C:\;Y^2 Z=X^3+aX^2 Z+bX Z^2\,.$$ 
Assume now that $\tilde{A}=(x_0,y_0,1)$ 
is a point on the cubic $\C$, where $x_0,y_0\in\F$ and $y_0\neq 0$. 
Then the point $(1,1,1)$ is on the curve
$$y_0^2\,Y^2 Z=x_0^3\,X^3+a\,x_0^2\,X^2 Z+b\,x_0\,X Z^2\,.$$ Now, 
by exchanging $X$ and $Z$ ({\ie},
$(X,Y,Z)\mapsto (Z,Y,X)$), dehomogenising with respect to the third 
coordinate ({\ie}, $(Z,Y,X)\mapsto (Z/X,\,Y/X,\,1)$), and multiplying
with $1/y_0^2$, we obtain that the point $A=(1,1)$ is on the 
curve $$\Cabc:\;y^2 x=\alpha+\beta x+\gamma x^2\,,$$ where
$\alpha,\beta,\gamma\in\F$.
Notice that since $A=(1,1)$ is on $\Cabc$, we have $\alpha+\beta+\gamma=1$. 
We denote this projective transformation which maps $\Gamma_{a,b}$ to
$\Gamma_{\alpha,\beta,\gamma}$ and $\tilde A$ to $A$ by $\Phi$.

In homogeneous coordinates, the neutral element of $\Cabc$ is 
$\nO=(0,1,0)$, and the image under $\Phi$ of the point $(0,0,1)$ on $\Cab$ is
$T=(1,0,0)$. With respect to the curve $\Cabc$, we obtain that the conjugate
$\konj{P}$ of a point $P=(x_0,y_0)$ on $\Cabc$ is of the form
$\konj{P}=(x_1,-y_0)$ for some $x_1\in\F$.

A generalisation of these observations is given by the following:

\begin{lem}\label{lem:trans}
Let\/ $\tilde{A}_0=(x_0,y_0)$ be a point on the 
cubic\/ $\Cab:y^2=x^3+ax^2+bx$, where\/ $x_0,y_0,a,b\in\F$ and\/ $y_0\neq 0$.
Then there exists an $\F$-projective transformation\/ $\Phi$ 
which maps the curve\/ $\Cab$ to the curve\/ 
$$\Cabcd:\;y^2 (x-\delta)=\alpha+\beta x+\gamma x^2
\qquad\text{(with\/ $\alpha,\beta,\gamma,\delta\in\F$),}$$ 
and the point\/ $\tilde{A}_0$ to~$A_0=(1,1)$. Moreover, we can require that\/
$\konj{A}_0=(-1,-1)$.
\end{lem}

\begin{proof}
By the above observations there exists an $\F$-projective transformation $\Phi$ 
which maps the curve $\Cab$ to the curve 
$$\tilde{\C}:\;y^2 x=\tilde{\alpha}+\tilde{\beta} x+\tilde{\gamma} x^2\qquad
\text{(with\,\ $\tilde{\alpha},\tilde{\beta},\tilde{\gamma}\in\F$),}$$ 
and the point $\tilde{A}_0$ to~$A_0=(1,1)$. The conjugate $\konj{A}_0$ of $A_0$ 
is of the form $\konj{A}_0=(x_1,-1)$, and by shifting and stretching the $x$-axis,
we obtain the curve 
$$\Cabcd:\;y^2 (x-\delta)=\alpha+\beta x+\gamma x^2\qquad
\text{(with\,\ $\alpha,\beta,\gamma,\delta\in\F$),}$$ 
which contains the points $A_0=(1,1)$ and $\konj{A}_0=(-1,1)$.
\end{proof}

With respect to the curve $\Cabcd$, we can compute the conjugate
of a point by the following:

\begin{fct}\label{fct:konj}
Let $P=(x_0,y_0)$ be a point on $\Cabcd$. Then
$$\konj{P}=\left(\frac{\alpha+\delta(x_0\gamma+\beta)}{\gamma(x_0-\delta)},\,-y_0\right).$$
\end{fct}

\begin{proof}
Let $P=(x_0,y_0)$ be a point on $\Cabcd$. Then
$$y_0^2(x_0-\delta)=\alpha+\beta x_0+\gamma x_0^2\,,$$ which implies that
$x_0$ is a root of $$x^2\gamma+x(\beta-y_0^2)+(\alpha-\delta y_0^2)\,,$$
and since the other root is $\frac{\alpha+\delta(x_0\gamma+\beta)}{\gamma(x_0-\delta)}$, we obtain
$\konj{P}=\left(\frac{\alpha+\delta(x_0\gamma+\beta)}{\gamma(x_0-\delta)},\,-y_0\right)$.
\end{proof}


Let $\Cab:\;y^2=x^3+a x^2+bx$ be a regular curve over some field $\F$ with torsion group $\Z/2n\Z$ 
(for some $n\ge 5$).
Each element of the group $\Z/2n\Z=\{0,1,\ldots,2n-1\}$ corresponds to a point on
$\Cab$. Let $\tilde{T}$ be the unique point of order~$2$. Then $\tilde{T}$ corresponds to~$n$.
Furthermore, let $A'$ be a point on $\Cab$ which corresponds
to~$1$. Then $A'$ is of order~$2n$. Finally, let $B'$
be the point on $\Cab$ which corresponds to~$2$. 
Then $A'+A'=B'$. 
Now, by {\LEM}\;\ref{lem:trans}, there is a
projective transformation $\Phi$ which maps the curve $\Cab$ to the curve $\Cabcd$, 
the point $A'$ to the point $A=(1,1)$, and the point $\konj{A'}$ to the 
point $\konj{A}=(-1,-1)$. Moreover, since 
$A+\konj{A}$ corresponds to~$1+(n+1)=n+2$, we obtain that $A+\konj{A}=\konj{B}$.
In other words, $\dritt{A}{\konj{A}}=-\konj{B}$, which implies that 
$-\konj{B}$ is on the line $A\konj{A}$. Hence, $-\konj{B}=(u,u)$ for some $u\in\F$,
and therefore $B=(v,u)$ for some $v\in\F$.

Since the points $A,\konj{A},B,\konj{B}$ belong to the curve $\Cabcd$, we obtain
$$\alpha=-u\,,\quad \beta=1\,,\quad 
\gamma=\frac{u^2-1}{u+v}\,,\quad 
\delta=u-\gamma\,,\quad
v=\frac{u^2-u\gamma-1}{\gamma}\,.$$

For $u,v\in\F$, let $$\Cuv:\ y^2\left(x-u+\frac{u^2-1}{u+v}\right)
=-u+x+\frac{u^2-1}{u+v}\,x^2\,.$$

By applying $\Phi^{-1}$ to the curve $\Cuv$, we obtain the curve $\Cab$ with
$$a=-2+3u^2+2u^3v+v^2\qquad \text{and}\qquad b=(u^2-1)^3(v^2-1)\,.$$

In the following sections we shall apply this approach to elliptic curves with torsion groups
$\Z/2n\Z$ for $n=5,6,7,8$.


\section{Elliptic Curves with Torsion Group $\boldsymbol{\Z/10\Z}$}

To warm up, we give a parametrisation of elliptic curves with torsion group
$\Z/10\Z$. 

Let $\Cab:\;y^2=x^3+a x^2+bx$ be a regular curve with torsion group $\Z/10\Z$
over~$\Q$.
Each element of the group $\Z/10\Z=\{0,1,\ldots,9\}$ corresponds to a rational point on
$\Cab$. Let $\tilde{T}$ be the unique point of order~$2$. Then $\tilde{T}$ correspond to~$5$.
Furthermore, let $\tilde{A}$ and $\tilde{B}$ be the rational points on $\Cab$ which correspond
to~$1$ and~$2$, respectively. Then $\tilde{A}$ is of order~$10$ and $\tilde{B}$ is of order~$5$. 
Finally, let $\Phi$ be a projective transformation $\Phi$ which maps the curve $\Cab$ to 
the curve $\Cabcd$, the point $\tilde{A}$ to the point $A=(1,1)$, and the conjugate of $\tilde{A}$
to the point $\konj{A}=(-1,-1)$. 
Let $B:=\Phi(\tilde{B})$ and $T:=\Phi(\tilde{T})$. Then, for $A,-A,\konj{A},\ldots$ we
obtain the following correspondence between these points on $\Cabcd$ and the elements of
the group $\Z/10\Z$:

\begin{center}
\begin{tabular}{l p{.7cm} p{.7cm} p{.7cm} p{.7cm} p{.7cm} p{.7cm} p{.7cm} p{.7cm} p{.7cm} p{.7cm}}
Elements of $\Z/10\Z${\Large${\mathstrut}$}&\hfill$0$\hfill&\hfill$1$\hfill&\hfill$2$\hfill&\hfill$3$\hfill&
\hfill$4$\hfill&\hfill$5$\hfill&\hfill$6$\hfill&\hfill$7$\hfill&\hfill$8$\hfill&\hfill$9$\hfill\\\hline
Points on $\Cabcd${\Large${\mathstrut}$}&\hfill$\nO$\hfill&\hfill$A$\hfill& 
\hfill$B$\hfill&\hfill$-\konj{B}$\hfill&\hfill$-\konj{A}$\hfill&\hfill$T$\hfill&\hfill$\konj{A}$\hfill&
\hfill$\konj{B}$\hfill&\hfill$-B$\hfill&\hfill$-A$\hfill 
\end{tabular}
\end{center}
\pagebreak

By definition, we have:
\begin{enumerate}[label=(\roman*)]
\item\label{itm:i} The points $A,\konj{A},-\konj{B}$ are collinear.
\item\label{itm:ii} The points $A,B,\konj{B}$ are collinear.
\end{enumerate}

Since $A=(1,1)$ and $\konj{A}=(-1,-1)$, by~\ref{itm:i} we have $-\konj{B}=(u,u)$ for 
some $u\in\Q$, {\ie}, $\konj{B}=(u,-u)$ and $B=(u,v)$. So, by~\ref{itm:ii}, we have 
$$v=\frac{3u-u^2}{u+1}\,,$$ and since $v=\frac{u^2-u\gamma-1}{\gamma}$, we have
$$\gamma=\frac{(u+1)^2(u-1)}{4u}\,.$$

Now, by applying the formulae $a=-2+3u^2+2u^3v+v^2$ and $b=(u^2-1)^3(v^2-1)$
for the parameters of the curve $\Cab$, we obtain the following result.

\begin{thm}\label{thm:mainZ10}
Let\/ $u\in\Q\setminus\{0,\pm 1\}$ and let 
\begin{eqnarray*}
a_1=-2\cdot(1 + 2 u - 5 u^2 - 5 u^4 - 2 u^5 + u^6)\,,&&
b_1=(u^2 - 1)^5\cdot (-1 - 4 u + u^2)\,.
\end{eqnarray*}
Then, the curve\/ $$\C_{a_1,b_1}:y^2=x^3+a_1 x^2+b_1 x$$
is an elliptic curve with torsion group\/ $\Z/10\Z$. Conversely, if\/ 
$\Cab$ is a regular elliptic curve with torsion group\/ $\Z/10\Z$, then there
exists a\/ $u\in\Q$ such that\/ $\Cab$ is isomorphic
to\/ $\C_{a_1,b_1}$. 
\end{thm}

\noindent{\bf Remarks.}
\begin{itemize}
\item In~\cite[Table\;3,\,p.\,217]{Kubert}, Kubert gives the following parametrisation of 
curves of the form $$y^2+(1-c)xy-by=x^3-bx^2$$ with torsion group $\Z/10\Z$
(see also Kulesz~\cite[p.\,341,(1.1.9)]{Kulesz}, who found Kubert's pa\-ra\-me\-tri\-zation in a different way):
$$\tau=\frac pq\,,\quad d=\frac{\tau^2}{\tau-(\tau-1)^2}\,,\quad c=\tau(d-1)\,,\quad b=cd\,.$$

\noindent After transforming Kubert's curve into the form
$$y^2=x^3+\tilde{a}x^2+\tilde{b}x\,,$$
we find $$\tilde{a}=
-(2 p^2 - 2 p q + q^2) (4 p^4 - 12 p^3 q + 6 p^2 q^2 + 2 p q^3 - 
   q^4)\,,\qquad\tilde{b}= 16 p^5 (p - q)^5 (p^2 - 3 p q + q^2)\,.$$
Now, by substituting in $\tilde{a}$ and $\tilde{b}$ the values $p$ 
and $q$ with $p+q$ and $2q$, respectively, and setting $u=\frac pq$,
we obtain $4a_1$ and $16b_1$, respectively, which shows that 
the two parametrisations are equivalent.

\item 
Recall that the Calkin-Wilf sequence
$$
s_1=1, \quad s_{n+1}=\frac{1}{2\lfloor s_n\rfloor-s_n+1}
$$
lists every positive rational number exactly once.
By checking the first 22\,000 fractions in this sequence
we found, with the help of {\tt MAGMA}, $46$~elliptic curves 
with torsion group $\Z/10\Z$ and rank~$3$. 

\item The following table gives 
the fractions $p/q$ and their indices in 
the Calkin-Wilf sequence of six of the $25$~known elliptic curves 
with torsion group $\Z/10\Z$ and rank~$4$ (see~\cite{dujella-web}):

\begin{center} 
\begin{tabular}{rrrl}
$p$\hspace*{1.5ex}{\,} & $q$\hspace*{1.5ex}{\,} & Calkin-Wilf index & discovered by\\
\hline
$2244$\LARGE{\mathstrut} & $1271$ & $307\,485$ & Fisher (2016)\\
$3051$ & $2164$ & $623\,897$           & Fisher (2016)\\
$4777$ & $7725$ & $1\,629\,610$        & Fisher (2016)\\
$1333$ & $475$  & $3\,137\,659$        & Dujella (2005)\\
$2407$ & $308$  & $67\,161\,983$       & Dujella (2008)\\
$1564$ & $1991$ & $532\,575\,944\,622$ & Elkies (2006)
\end{tabular}
\end{center}

\item For $p/q = 13360/9499$ (Calkin-Wilf index $15\,352\,857$), we identified a new 
$\Z/10\Z$ curve of conditional rank $4$:
	\begin{align*}
		y^2	&+xy=x^3-37727175946500513344407792867239647500428495062395\,x\\
			&+89192287551119365352685730014597522447709983914961211487019041502289387025
	\end{align*}
{\tt MAGMA} computations reveal three generators on the curve and confirm that its root number is $1$, 
therefore the rank should be even. The point search up to height $2^{38}$ on each of the $256$ $4$-coverings 
for each of the $4$ curves in the isogeny class did not uncover the missing last generator. We leave it as 
an open challenge to test new descent methods.
\end{itemize}
 

\section{Elliptic Curves with Torsion Group $\boldsymbol{\Z/12\Z}$}

Let us now consider parametrisations of elliptic curves with torsion group
$\Z/12\Z$. By similar arguments as above, one can show the following result.

\begin{thm}\label{thm:mainZ12}
Let\/ $t\in\Q\setminus\{1\}$ be a positive rational and let 
$$\begin{array}{lcl}
a_1=2\left(3 t^8 + 24 t^6 + 6 t^4 -1\right),&&
b_1=(t^2-1)^6 (1 + 3 t^2)^2\,.
\end{array}$$
Then the curve\/ $$\C_{a_1,b_1}:y^2=x^3+a_1 x^2+b_1 x$$
is an elliptic curve with torsion group\/ $\Z/12\Z$. Conversely, if\/ 
$\Cab$ is a regular elliptic curve with torsion group\/ $\Z/12\Z$, then there
exists a positive rational\/ $t$
such that\/ $\Cab$ is isomorphic to\/ $\C_{a_1,b_1}$. 
\end{thm}

In~\cite[Table\;3,\,p.\,217]{Kubert}, Kubert gives the following parametrisation of 
elliptic curves of the form $$y^2+(1-c)xy-by=x^3-bx^2$$ with torsion group $\Z/12\Z$
(see also Kulesz~\cite[p.\,341,(1.1.10)]{Kulesz}, who found Kubert's parametrisation 
in a different way):
\begin{equation}\label{KubertForm}
\tau=\frac rs\,,\quad m=\frac{3\tau-3\tau^2-1}{\tau-1}\,,\quad f=\frac m{1-\tau}\,,\quad
d=m+\tau\,,\quad c=f(d-1)\,,\quad b=cd\,.
\end{equation}

\noindent After transforming Kubert's curve into the form
$$y^2=x^3+\tilde{a}x^2+\tilde{b}x\,,$$
we find 
\begin{align*}
\tilde{a}&=
s^8 + 12 r (r - s)\left(s^6 + 2 r (r - s) (r^2 - r s + s^2) (r^2 - r s + 2 s^2)\right)\,,\\[1.2ex]
\tilde{b}&=16 r^6 (r - s)^6\left(3 r (r - s) + s^2\right)^2\,.
\end{align*}
Now, for $t=\frac rs$ we obtain 
$$\begin{array}{lcl}
a_1:=6r^8 + 48r^6s^2 + 12r^4s^4 - 2s^8\qquad\text{and}\qquad
b_1:=(r^2 - s^2)^6\,(3 r^2 + s^2)^2\,. 
\end{array}$$
Then, by substituting in $\tilde{a}$ and $\tilde{b}$, $r$ with $r+s$ and $s$ with $2s$, 
we obtain $4a$ and $16b$, respectively. This shows that the two elliptic curves
$$\C_{\tilde{a},\tilde{b}}:\ y^2=x^3+\tilde{a}x^2+\tilde{b}x$$ and 
$$\Cab:\ y^2=x^3+ax^2+bx$$ are equivalent.

\subsection{Elliptic Curves of Rank at Least 2}
By checking the first 3441 fractions $r/s$ of the Calkin-Wilf sequence
we found, with the help of {\tt MAGMA}, $125$ fractions which lead to
elliptic curves with torsion group $\Z/12\Z$ and rank~$2$,
and among these 3441 fractions, we even found some which lead to
curves with rank~$3$.
%

As a matter of fact, we would like to mention that until today (November\;2021),
up to isomorphisms only one elliptic curve of rank~$4$ is known, namely 
\begin{multline*}
y^2 + xy = x^3 - 4422329901784763147754792226039053294186858800\,x\\ 
         + 98943710602886706347390586357680210847183616798063680624530387016000
\end{multline*}   
discovered by Fisher in~2008 (see~\cite{dujella-web}). 
This curve is isomorphic to 
\begin{multline*}
y^2 = x^3 +588436986469809874425598\,x^2\\
+ 44662083920000859376188675997725867856489478401\,x
\end{multline*}
which is of type $\C_{a,b}$, where $r=726$ and $s=133$ 
(Calkin-Wilf index of $726/133$ is $274\,335$).

\subsection{Families of Elliptic Curves with Positive Rank}

In this section, we construct three infinite families of elliptic curves $\Cab$ with positive
rank. Other such families were found, for example, by Rabarison\;\cite[Thm.\,12]{Rabarison_family},
Kulesz\;\cite[Thm.\,2.12]{Kulesz} (see also \cite[Sec.\,2.12]{Kulesz_family}),  
and by Suyama (see~\cite[p.\,262\,{f}]{Montgomery}).
Although the parametric families of positive rank and torsion group $\Z/12\Z$ 
are not explicitly given in the work of Rabarison, 
they are mentioned on page~17, line~3 of his manuscript 
and on page~90, line~1 of his thesis. 
In fact, the elliptic curves which correspond to our three families 
are, in Cremona's notation, the curves 368d1, 226a1 and~720e2.

Let $t\in\Q\setminus\{-1,0,1\}$. Instead of $\Cab$ we consider 
the equivalent elliptic curve $$\Ct:\ y^2=x^3+
2\left(3 t^8 + 24 t^6 + 6 t^4 -1\right)x^2+(t^2-1)^6 (1 + 3 t^2)^2\,x\,.$$
The finite torsion points of $\Ct$ are given in the following table:

\begin{center}
\begin{tabular}{ccc}
Order & $x$-coordinate & $y$-coordinate\\
\hline
$2$\LARGE{\mathstrut} & $0$ & $0$\\[1ex]
$3$ & $(t^2-1)^4$ & $\pm\,4t^2\, (t^2-1)^4\, (t^2+1)$\\[1ex]
$4$ & $-(t^2-1)^3\, (1+3t^2)$ & $\pm\,8t^3\, (t^2-1)^3\, (1+3t^2)$\\[1ex]
$6$ & $(t^2-1)^2\, (1+3t^2)^2$ & $\pm\,4t^2\, (t^2-1)^2\, (t^2+1)\, (1+3t^2)^2$\\[1ex]
$12$ & $(t^2-1)\, (t+1)^4\, (1+3t^2)^2$ & $\pm\,4t\, (t^2-1)\, (t+1)^4\, (1+t^2)\, (1+3t^2)$\\[1ex]
$12$ & $-(t^2-1)\, (t+1)^4\, (1+3t^2)^2$ & $\pm\,4t\, (t^2-1)\, (t-1)^4\, (1+t^2)\, (1+3t^2)$
\end{tabular}
\end{center}

Now, if we find any additional rational point $P$ on $\Ct$, then the order of $P$
is infinite which implies that the rank of $\Ct$ is positive. On the other hand, 
if we find an infinite family $\nT$ of values for $t$ such that for every $t\in\nT$,
the curve $\Ct$ has an additional point, then $\{\Ct:t\in\nT\}$ is an infinite family 
of elliptic curves with torsion group $\Z/12\Z$ and positive rank.\medskip

\subsubsection{First Family} 

Let $P_1=(x_1,y_1)$ with 
$$x_1=-(t+1)^2(t-1)^6\,.$$
Then $P_1$ is a rational point on $\Ct$ if and only if
$$v^2=-(t^4+8t^3+2t^2+1)\quad\text{for some rational $v$.}$$
This quartic curve has a rational solution $(t,v)=(-1,2)$, hence,
by \cite[p.\,472]{DujellaBook}, 
it is equivalent to the elliptic curve
	\begin{equation*}
		y^2=x^3+x^2-1\,,
	\end{equation*}
	which is a rank-$1$ elliptic curve, where $G_1:=(1,1)$ is a point of
	infinite order. In particular, for all but finitely many $k\in\Z$, 
	for $[k]G_1=(x_k,y_k)$ and $t_k:=(y_k-1)/(2x_k-y_k-1)$, 
	$\C_{t_k}$ is a non-singular curve with torsion
	group $\Z/12\Z$ and positive rank.\smallskip
	
	\noindent{\bf Examples.} For $k=2$ we obtain $[2]G_1=(13/4, -53/8)$
	and $t_2=-61/97$. So, the parameters $a$ and $b$ of $\Cab$ are
	$$a=\frac{23452774585480768}{7837433594376961}\qquad\text{and}\qquad
	b=\frac{14332124021409323029654935699456}{61425365346268570446197767595521}\,,$$
	where $\Cab$ has rank~$2$.
	In order to compute the rank of $\Cab$, it seems to be faster to use Kubert's
	form\;(\ref{KubertForm}) with $\tau=(r+s)/(2s)$ where $t_2=r/s$, which gives us
	$$1-c=\frac{53471797}{47824783}\qquad\text{and}\qquad
	b= -\frac{37072646910}{366481312129}\,.$$
	
	The following table summarizes what we have found with the help of {\tt MAGMA}.
	
	\begin{center} 
		\begin{tabular}{ccccc}
			{\;}\hspace*{2ex}$k$\hspace*{2ex}{\;} & $t_k$ & $1-c$ & $b$ & Rank\\
			\hline
			$-2$\LARGE{\mathstrut}   & $-5$ & $-\tfrac{163}{27}$ & $\tfrac{2470}{81}$ & $1$\\[2.5ex]
			$-1$ & $-1$              & $1$ & $0$ & $\C_{-1}$ is singular\\[2.5ex]
			$2$ & $-\tfrac{61}{97}$          & $\tfrac{53471797}{47824783}$ 
			& $-\tfrac{37072646910}{366481312129}$ & $2$\\[2.5ex]
			$3$ & $-\tfrac{5737}{1921}$      & $-\tfrac{172463134332983}{107840890126669}$ 
			& $\tfrac{5130041565973335306660}{793224637894724969521}$ & $\ge 1$\\[2.5ex]
			$4$ & $-\tfrac{1500953}{1090945}$ & 
			$\tfrac{1763009864383862432547449}{2374469464172162335214805}$
			& $\tfrac{1052634091165646643245217267815958652}{3357046492915255175591448074492123025}$ 
			& $\ge 1$
		\end{tabular}
	\end{center}
{\tt MAGMA} calculations also confirm the rank to be at least 1 for both $k=5$ and $k=6$.

\subsubsection{Second Family} 

Let $P_2=(x_2,y_2)$ with $$x_2=(t+1)^8\,.$$
Then $P_2$ is a rational point on $\Ct$ if and only if
$$v^2=t^4-2t^3+13t^2+4t+4\quad\text{for some rational $v$.}$$ 
This quartic curve has a solution $(t,v)=(0,2)$ and, by \cite[Thm.\,2.17]{Washington}, 
is birationally equivalent to
	\begin{equation*}
		y^2 + xy = x^3 - 5x + 1\,,
	\end{equation*}
	which is, according to {\tt MAGMA}, a rank-$1$ elliptic
	curve, where $G_2:=(0,1)$ is a point of
	infinite order. In particular, for all but finitely many $k\in\Z$, 
	for $[k]G_2=(x_k,y_k)$ and $t_k:=2(x_k+2)/(y_k+1)$, 
	$\C_{t_k}$ is a non-singular curve with torsion
	group $\Z/12\Z$ and positive rank.\smallskip

\vbox{
\noindent{\bf Examples.} With the help of {\tt MAGMA}, we computed the
rank of the curve $\C_{t_k}$ for $k=1,\ldots,8,10$:
		
\begin{center} 
\begin{tabular}{ccccc}
{\;}\hspace*{2ex}$k$\hspace*{2ex}{\;} & $t_k$ & $1-c$ & $b$ & Rank\\
\hline
$1$\LARGE{\mathstrut} & $2$ & $79$ & $390$ & $1$\\[3ex]
$2$ & $\frac{4}{3}$              & $533$ & $\frac{13300}{3}$ & $1$\\[3ex]
$3$ & $\frac{3}{14}$    
& $\frac{7261}{18634}$ 
& $\frac{2331465}{2869636}$ 
& $1$\\[3ex]
$4$ & $\frac{175}{17}$      
& $\frac{395470463}{8381663}$ 
& $\frac{5983231581600}{11256573409}$ 
& $1$\\[3ex]
$5$ & $-\frac{799}{1200}$ 
& $\frac{3553536961599}{3195202399600}$ 
& $-\frac{744762911993283599}{7664651516160480000}$ 
& $1$\\[3ex]
$6$ & $-\frac{43872}{18847}$ 
& $-\frac{2079664621920150527}{4649857101108754273}$ 
& $\frac{15343052413669431178634306400}{5496433301673119911952465089}$ 
& $1$
\end{tabular}
\end{center}
{\tt MAGMA} calculations also confirm that the rank is exactly 1 for $k=7$ and $k=10$, and the rank is at least 1 for $k=8$.}

\subsubsection{Third Family}

Let $P_3=(x_3,y_3)$ with $$x_3=\tfrac{3}{4}(t+1)^4(t-1)^4\,.$$
Then $P_3$ is a rational point on $\Ct$ if and only if
$$v^2=75t^4+66t^2+3\quad\text{for some rational $v$.}$$
This quartic curve has a solution $(t,v)=(1,12)$, hence it is equivalent to an elliptic curve
	\begin{equation*}
		y^2 = x^3 - 147x - 286
	\end{equation*}
	of rank~$1$, generated by the point $G_3=(-5,-18)$ of infinite order, 
	as determined by {\tt MAGMA}. For all but finitely many $k\in\Z$, 
	for $[k]G_3=(x_k,y_k)$ and $t_k:=(y_k-3x_k+3)/(y_k+9x_k+63)$, 
	$\C_{t_k}$ is a non-singular curve with torsion
	group $\Z/12\Z$ and positive rank.

We have computed the rank of the curve $\C_{t_k}$ for $k=-2,2,3,4$:
\begin{center} 
	\begin{tabular}{ccccc}
		{\;}\hspace*{2ex}$k$\hspace*{2ex}{\;} & $t_k$ & $1-c$ & $b$ & Rank\\
		\hline
		$-2$\LARGE{\mathstrut} & $-\frac{1}{11}$ & $\frac{2531}{2376}$ & $-\frac{9455}{156816}$ & $2$\\[2.5ex]
		%
		%
		%
		%
		$2$ & $-\frac{59}{169}$              & $\frac{282022931}{250380936}$ 
		                                     & $-\frac{506936401895}{4823839112976}$ & $2$\\[2.5ex]
		$3$ & $-\frac{9059}{61}$    
		& $-\frac{2502776666788081}{5783947776000}$ 
		& $\frac{102937783902951852766681}{1608862913372160000}$ 
		& $\ge 1$\\[2.5ex]
		$4$ & $\frac{2086379}{6069899}$      
		& $-\frac{58188154268169008260521361}{47961469780361881818624000}$ 
		& $\frac{2186504518566993279256742865370434477481}{579843715590440818015794769800437760000}$ 
		& $\ge 1$
	\end{tabular}
\end{center}


\section{Elliptic Curves with Torsion Group $\boldsymbol{\Z/14\Z}$}

For an elliptic curve over some field $\F$ with torsion group $\Z/14\Z$, 
starting with a value for $u\in\F$, we compute a value 
for $v$, which will lead to a parametrisation of elliptic curves with 
torsion group $\Z/14\Z$.


\begin{thm} Let\/ $\F$ be a field containing\/ $\Q$. Then there exists an
elliptic curve\/ $\Cab$ over\/ $\F$ with torsion group\/ $\Z/14\Z$ if and only if
for some\/ $u\in\F\setminus\{-1,0,1\}$, 
$$\sqrt{\mathstrut 1 - 2 u + u^2 + 4 u^3}$$
belongs to\/ $\F$.
\end{thm}

\begin{proof}
Assume that the curve $\Cabcd$ over $\F$ has torsion group $\Z/14\Z$ and that
the points $A=(1,1)$, $\konj{A}=(-1,-1)$, $B=(v,u)$, $-\konj{B}=(u,u)$ belong 
to $\Cabcd$, where $A$, $\konj{A}$, $B$, $-\konj{B}$ correspond to~$1$, $8$, $2$, $5$, 
respectively. Finally, let $C:=A+B$. Then $C$ corresponds to~$3$. The complete group 
table is given by the following table:

\begin{center}
\begin{tabular}{l p{.7cm} p{.7cm} p{.7cm} p{.7cm} p{.7cm} p{.7cm} p{.7cm} p{.7cm}}
Elements of $\Z/14\Z${\Large${\mathstrut}$}&\hfill$0$\hfill&\hfill$1$\hfill&\hfill$2$\hfill&\hfill$3$\hfill&
\hfill$4$\hfill&\hfill$5$\hfill&\hfill$6$\hfill&\hfill$7$\hfill\\\hline
Points on $\Cabcd${\Large${\mathstrut}$}&\hfill$\nO$\hfill&\hfill$A$\hfill& 
\hfill$B$\hfill&\hfill$C$\hfill&\hfill$-\konj{C}$\hfill&\hfill$-\konj{B}$\hfill&\hfill$-\konj{A}$\hfill&
\hfill$T$\hfill\\[3ex]
Points on $\Cabcd${\Large${\mathstrut}$}
&\hfill$\nO$\hfill&\hfill$-A$\hfill& 
\hfill$-B$\hfill&\hfill$-C$\hfill&\hfill$\konj{C}$\hfill&\hfill$\konj{B}$\hfill&\hfill$\konj{A}$\hfill&
\hfill$T$\hfill\\\hline
Elements of $\Z/14\Z${\Large${\mathstrut}$}
&\hfill$14$\hfill&\hfill$13$\hfill&\hfill$12$\hfill&\hfill$11$\hfill&
\hfill$10$\hfill&\hfill$9$\hfill&\hfill$8$\hfill&\hfill$7$\hfill
\end{tabular}
\end{center}
Since $\dritt B{\konj{B}}=C$, the point $C$ is on the line $g$ passing through $B$ and $\konj{B}$. 
Furthermore, we have $\dritt AB=-C=\dritt{\konj{A}}{\konj{B}}$. In other words,
$-C$ is the intersection point of the lines $AB$ and $\konj{A}\konj{B}$, denoted
$-C=AB\wedge\konj{A}\konj{B}$. Furthermore, we have $-\konj{C}=A\konj{B}\wedge\konj{A}B$.
Notice that the points $C$ and $-\konj{C}$ have the same $y$-coordinate. 

In homogeneous coordinates, we obtain
$$g:\ B\times\konj{B}=(v,u,1)\times (u,-u,1)=(2 u,\,u - v,\,-u^2 - u v)$$
and
$$-C\;=\ (A\times B)\times (\konj{A}\times\konj{B})=
(-3 u - u^2 + v + 3 u v,\,(u-1) (3 u - v),\,(1-u) (u + v)),$$
and therefore $$C=(-3 u - u^2 + v + 3 u v,\,(1-u) (3 u - v),\,(1-u) (u + v))\,.$$ 
Since the point $C$ belongs to $g$, we must have the scalar product $\langle g,C\rangle=0$, {\ie},
$$u^2 (-3 - 6 u + u^2)+2 u (-1 + 4 u + u^2)\,v+ (u-1)^2\,v^2=0\,.$$
Now, for $u\in\F$, this implies that also $$v_{1,2}=
\frac{u\left(1 - 4 u - u^2 \pm 2\sqrt{\mathstrut1 - 2 u + u^2 + 4 u^3}\right)}{(u-1)^2}$$ 
belong to $\F$, and hence $\sqrt{\mathstrut 1 - 2 u + u^2 + 4 u^3}$ belongs to $\F$.

On the other hand, if $\Cab$ is an elliptic curve over $\F$ with torsion group $\Z/14\Z$,
then we can transform this curve (over $\F$) to the curve $\Cabcd$ 
which contains the points $A=(1,1)$, $B=(v,u)$, and $-\konj{B}=(u,u)$ with the 
above properties. In particular, we have that $u,v\in\F$ which implies that
$\sqrt{\mathstrut 1 - 2 u + u^2 + 4 u^3}$ belongs to~$\F$.
\end{proof}

As an immediate consequence we obtain:
\begin{cor}
Let\/ $u\in\Q\setminus\{-1,0,1\}$ and let\/ $d=1 - 2 u + u^2 + 4 u^3$. 
Then there exists an elliptic curve\/ $\Cab$ over\/ $\Q(\sqrt{d})$ 
with torsion group\/ $\Z/14\Z$.
\end{cor}

\subsection{High Rank Elliptic Curves with Torsion Group $\boldsymbol{\Z/14\Z}$}


By some further calculations, we can slightly simplify the formulae for 
the parameters $a$ and $b$ of the curve $\Cab$. 
For $z=\sqrt{\mathstrut 1 - 2 u + u^2 + 4 u^3}$ we have:
$$a=-2\left(1 - 4u + 2u^2 + 10u^3 - 18u^4 - 10u^6 + 2u^7 + u^8\right)
+ 4u^2\left(1 - 4u - 2u^3 + u^4\right)\,z$$
$$b=(1-u)^7(1+u)^3\left((1+u)(1 - 5u + 6u^2 + 6u^3 - 23u^4 - u^5)
- 4u^2(1 - 4u - u^2)\,z\right)$$

With the help of {\tt{MAGMA}} we found that for each
$u=4,\,4/7,\,4/9,\,8/5,\,-4/11,\,5/17,\,\frac{\sqrt{2081}+39}{8},$ 
$\frac{\sqrt{2713}+37}{32},\,\frac{\sqrt{12121}+121}{18},$
$\frac{\sqrt{23641}+109}{98},\,\frac{\sqrt{55441}+169}{128},$ 
the corresponding curve has rank~$2$.
We would like to mention that different values of $u\in\Q$ may not necessarily
lead to different quadratic fields $\Q(\sqrt{d\mathstrut})$. For example, for $u_1=4/9$ and
$u_2=5/13$, both curves have torsion group $\Z/14\Z$ over the same quadratic field
$\Q(\sqrt{265})$. Moreover, for $u_1$ and $u_2$, the corresponding curves have the same rank, which follows from the following result 
(see~Rabarison \cite[Lem.\,4.4]{Rabarison}).

\begin{fct}\label{fct:Z/6}
Let\/ $d$ be a square-free integer and let\/ $\F_d=\Q(\sqrt{d\mathstrut})$. Furthermore, let\/
$u_0\in\F_d$ be such that\/ $\F_d=\Q(\sqrt{z})$ where\/ $z=1 - 2u_0 + u_0^2 + 4u_0^3$. 
Then the elliptic curve $$\Gamma_0:\ y^2= 4x^3 +x^2-2x+1$$ 
has torsion group\/ $\Z/6\Z$ over\/ $\F_d$ for\/ 
$d\neq -7$ and has torsion group\/ $\Z/2\Z\times\Z/6\Z$ over\/ $\F_{-7}$.
\end{fct}

Now, since $(u_0,z)$ is a non-torsion point on 
$\Gamma_0$, the curve $\Gamma_0$ has rank~$\geq1$ over $\F_d$, and by adding (in the case $d\neq -7$) 
the~$6$ torsion points of $\Gamma_0$ to $(u_0,z)$, we obtain the following 
$6$~values for~$u$, which all lead to essentially the same curve with 
torsion group $\Z/14\Z$ over $\F_d$\,:
$$u_1=u_0\,,\qquad
u_2=\frac{1-u_0}{1+u_0}\,,\qquad
u_{3,4}=\frac{u_0(u_0+1)\pm z}{(u_0-1)^2}\,,\qquad
u_{5,6}=\frac{(1-u_0)\pm z}{2u_0^2}\,.
$$
For example, if $u_1=4/9$, then $u_2=5/13$ and the corresponding curves 
are essentially the same. 

In order to obtain different curves with torsion group $\Z/14\Z$,
we can, for example, start with an arbitrary $u_0\in\Q\setminus\{-1,0,1\}$ and 
double the point $(u_0,z)$ on $\Gamma_0$ (over the corresponding field $\F_d$).
This way, we get the following value for~$u$:
$$u=\frac{u_0(u_0-1) (u_0^2+u_0+2)}{z^{2\mathstrut}}=
\frac{u_0(u_0-1) (u_0^2+u_0+2)}{(u_0+1) (4u_0^{2\mathstrut}-3u_0+1)}.$$

For example, taking $u_0=1/2$ (curve of rank~$0$) we produce 
a different $\Z/14\Z$ curve with $u=-11/12$ (rank~$1$) over the
same field $\Q(\sqrt{3\mathstrut})$. 
Similarly, taking $u_0=2/3$ (curve of rank~$1$) we produce a different $\Z/14\Z$ 
curve with $u=-8/15$ (rank~$0$) over the same field $\Q(\sqrt{105\mathstrut})$.

Another approach to find independent values for~$u$ would be to search for values~$d$, 
such that $\Gamma_0$ has high rank over~$\F_d$. With the help of {\tt MAGMA} we found an 
abundance of quadratic fields $\F_d$ over which the $\Z/6\Z$ curve $\Gamma_0$ has rank~$2$--$5$. 
The fields $\F_{d}$ with the smallest absolute $d$-values for the curve $\Gamma_0$ of rank~$2, 3, 4, 5$ 
are $\F_{22}, \F_{874}, \F_{-5069}, \F_{1578610}$, respectively.
Two essentially different $\Z/14\Z$ curves over $\Q(\sqrt{22\mathstrut})$ we found this way
are produced by $u_1=1/8$ and $u_2=7/4$. The former curve has rank~$0$, whereas the latter 
has rank~$1$. The mentioned $u$-values correspond to the two generators of $\Gamma_0$, which
has rank~$2$ over~$\F_{22}$.
Similarly, the four $\Z/14\Z$ curves over $\Q(\sqrt{2233\mathstrut})$ produced by $u_1=4/7$ 
(curve of rank~2), $u_2=13$ (rank~1), $u_3=1/28$ (rank~0), and $u_4=-2/11$ (rank~1) 
are all essentially different, as all the $u$-values correspond to pairwise linearly 
independent combinations of the three generators of the curve $\Gamma_0$ over $\F_{2233}$.

Let us now turn back to the search of high rank elliptic curves with 
torsion group $\Z/14\Z$. With the help of {\tt MAGMA} we first found that for $u=11/5$ 
(or equivalently for $u=-3/8$), the produced curve has rank~$3$~\,--\,~the current 
record for torsion group $\Z/14\Z$ (see Dujella~\cite{DujellaQuad}). 
The curve in Weierstrass normal form is
$$y^2 = 
x^3 - \frac{64\left(214412 \sqrt{430} - 4876337 \right)}{390625}\,x^2 -
\frac{146767085568\left(9559 \sqrt{430} - 198202 \right)}{152587890625}\,x\,,
$$
and is isomorphic to the curve
$$y^2 = 
x^3 - (214412 \sqrt{430}-4876337)\,x^2 -
12^7(9559 \sqrt{430}-198202)\,x\,
$$
with the three independent points of infinite order
\begin{align*}
	P_1 &= (85536\sqrt{430} - 1844856, 117398160\sqrt{430} - 2391265800),\\[.8ex]
	P_2 &= (-45684\sqrt{430} + 945999, 197481240\sqrt{430} - 4096319040),\\[.8ex]
	P_3 &= (150336\sqrt{430} + 5895504, 389901600\sqrt{430} + 19611195120).
\end{align*}
\indent{}Later, we found that also $u = \frac{\sqrt{-2759}-11}{32}$ produces a $\Z/14\Z$ curve of rank~$3$:
\begin{align*}
	y^2 &= x^3 -(5327056844923892\sqrt{-2759} + 151212615362621956)\,x^2 \\
	%
		&+6525845768\,(236459153187683150165981\sqrt{-2759} - 27186277677196768482611999)\,x
\end{align*}
with the three generators
	\begin{align*}
		P_1 = (&3390432200076922\sqrt{-2759} + 362094129708044162,\\
		&4074194213434761922471680\sqrt{-2759} + 34672137787509115316417280),\\[.8ex]
		P_2 = (&2744055515797882\sqrt{-2759} + 421484961222088322,\\
		&3344617042430565258489600\sqrt{-2759} + 57335992491288398889081600),\\[.8ex]
		P_3 = (&29793656566415482\sqrt{-2759} - 1346439443735230078,\\
		-&17912501237343415639622400\sqrt{-2759} -		2714013162586144875488198400).
	\end{align*}

Another parametrisation of elliptic curves over a quadratic field
with torsion group $\Z/14\Z$ is given by Rabarison~\cite[Sec.\,4.2]{Rabarison}. 
The defining polynomial of the quadratic field is $w^2 + w u + w = u^3 -u$, 
which leads to $$w_{1,2}=\frac{-(u+1)\pm\sqrt{1-2u +u^2+4u^3}}{2}\,.$$
The parametrised curve is 
$$E_{\ta,\tb}:\ y^2 +\ta xy +\tb y=x^3+\tb x^2$$
with
\begin{align*}
\ta&=\frac{u^4-u^3w+u^2(2 w-4)-uw+1}
{(u+1)(u^3-2u^2-u+1)}\,,\\[1.4ex]
\tb&=\frac{u(1-u)\left(u^5 - u^4 - 2 u^3 w + u^2 + u(2w-1) - w\right)}{(u+1)^2(u^3-2u^2-u+1)^2}\,,
\end{align*}
where the point $(0,0)$ is a point of order~$14$.

Like Kubert's parametrisation for elliptic curves with torsion group $\Z/10\Z$ or $\Z/12\Z$, 
Rabarison's parametrisation for elliptic curves with torsion group $\Z/14\Z$ can be
transformed to our parametrisation. For this, notice first that $a$ and $b$ depend on $u$ and
$v$. Now, for $$v=\frac{u (1 - 4 u - u^2 + 2 z)}{(u - 1)^2}$$ 
we obtain expressions for $a$ and $b$ which just depend on $u$ and $z$.
On the other hand, if we set $$v=-\frac{u (u^2 + 2 u - 3 - 4 w )}{(u - 1)^2}\,,$$ then,
for $$w=\frac{-(u+1)+z}{2}\,,$$ we obtain 
exactly the same expressions for $a$ and $b$ (also depending just on $u$ and $z$).
A similar result we get for $\ta$ and $\tb$ by setting
$$w=\frac{u (u^2 + 2 u - 3) + v (u - 1)^2}{4 u}\,.$$

With respect to Rabarison's parametrisation, the curve given above with rank~$3$ 
obtained by $u=11/5$~is 
\begin{multline*}
{\ }\hspace{5ex}y^2 +\frac{-15\,835 + 792 \sqrt{430}}{1\,160}\,xy
+\frac{165 (-46\,394 + 2\,237 \sqrt{430})}{26\,912}\,y\\[2ex]
=\ x^3 + \frac{165 (-46\,394 + 2\,237 \sqrt{430})}{26\,912}\,x^2,\hspace{5ex}{\ }
\end{multline*}

where three independent points of infinite order are:
\begin{align*}
P_1 &= \biggl(
\frac{85983835575 - 4146565170 \sqrt{430}}{39891856},\\
&\qquad \frac{27225 (2574034836965503 - 124130941794256 \sqrt{430})}{423791610918272}
\biggr),\\[1.7ex]
P_2 &= \biggl(
\frac{2298171927997 - 110831541546 \sqrt{430}}{1777510688},\\
&\qquad \frac{121 (69150183657283025 - 3334717851516528\sqrt{430})}{105982297261312}
\biggr),\\[1.7ex]
P_3 &= \biggl(
\frac{-36778174426785 + 1824475255680 \sqrt{430}}{15921363350048},\\
&\qquad \frac{1089 (1274259861468796925 - 61481356412574002 \sqrt{430})}{89843234417066460928}
\biggr).
\end{align*} 

\subsection{A Normal\;Form for Elliptic\;Curves with \hbox{Torsion\;Group\,$\boldsymbol{\Z/14\Z}$}}

For $u\in\Q$, $d:=1 - 2 u + u^2 + 4 u^3$, $z:=\sqrt{d}$, and $$v:=
\frac{u\left(1 - 4 u - u^2 \pm 2\sqrt{\mathstrut 1 - 2 u + u^2 + 4 u^3}\right)}{(u-1)^2}\,,$$ 
the elliptic curve 
$$\C_{u,v}:\ y^2=(u^2-1)^2 (v^2-1)\cdot \frac 1x\;+\;(3 u^2 + 
2 u^3 v + v^2 -2)\;+\;(u^2-1)\cdot x$$
over the quadratic field $\Q(\sqrt{d})$ has torsion group $\Z/14\Z$. 
Notice that $\C_{u,v}$ has two points at infinity, namely $(0,1,0)$, which is
the point of order~$1$, and $(1,0,0)$, which is the point of order~$2$.
The finite torsion points of $\C_{u,v}$ are given in the following table.

\begin{center}
\begin{tabular}{ccc}
Order & $x$-coordinate & $y$-coordinate\\
\hline
$14$\LARGE{\mathstrut} & $u^2-1$ & $\pm u(u+v)$\\[3ex]
$7$ & $v^2-1$ & $\pm u(u+v)$\\[3ex]
$14$ & $-(u-1)(v-1)$ & $\pm(u+v)$\\[3ex]
$7$ & $-(u+1)(v+1)$ & $\pm(u+v)$\\[3ex]
$14$ & $\displaystyle -\frac{(u+1)^2 (v-1)}{(u-1)^{\mathstrut}}$ & $\pm (3u-v)$\\[3ex]
$7$ & $\displaystyle -\frac{(u-1)^2 (v+1)}{(u+1)^{\mathstrut}}$ & $\pm (3u-v)$
\end{tabular}
\end{center}
As a matter of fact we would like to mention that for
$$a:=(u^2-1)^2 (v^2-1)\,,\quad
b:=3 u^2 + 2 u^3 v + v^2 -2\,,\quad 
c:=(u^2-1)\,,$$
if $(x_0,y_0)$ is a point on $\C_{u,v}$, then 
$\left(a/x_0,\,a y_0/x_0\right)$ 
is a point on $$\C_{b,\hspace{1pt}ac}:\ y^2=x^3\,+\,b\,x^2\,+\,ac\,x\,,$$
where the point at infinity $(1,0,0)$ is moved to $(0,0)$.


\section{Elliptic Curves with Torsion Group $\boldsymbol{\Z/16\Z}$}

In this section we use our geometric approach to construct a 
parametrisation of elliptic curves with torsion group $\Z/16\Z$.


\begin{thm}\label{thm:mainZ/16}
 Let\/ $\F$ be a field containing\/ $\Q$. Then there exists an
elliptic curve\/ $\Cab$ over\/ $\F$ with torsion group\/ $\Z/16\Z$ if and only if
for some\/ $\alpha\in\F\setminus\{-1,0,1\}$, 
$$\alpha\sqrt{1-\alpha^{\mathstrut 2}}\pm\sqrt{\alpha\,(\alpha^2-1)\Bigl(1+\sqrt{1-\alpha^2}-
\alpha\bigl(1 + \alpha +\sqrt{1-\alpha^2}\,\bigr)\Bigr)}$$
or
$$\alpha\sqrt{1-\alpha^{\mathstrut 2}}\pm\sqrt{\alpha\,(\alpha^2-1)\Bigl(1-\sqrt{1-\alpha^2}-
\alpha\bigl(1 + \alpha-\sqrt{1 - \alpha^2}\,\bigr)\Bigr)}$$
belongs to\/ $\F$.
\end{thm}

\begin{proof}
Assume that the curve $\Cabcd$ over $\F$ has torsion group $\Z/16\Z$ and that
for the points $A$ and $\konj{A}$, which
correspond to~$2$ and $10$, respectively, we have $A=(1,1)$ and $\konj{A}=(-1,-1)$. 
Furthermore, let $B$ and $D$ be points on
$\Cabcd$ which correspond to $7$ and $4$, respectively. Then,
the group table with respect to these points, their inverses and their conjugates,
is given by the following table:

\begin{center}
\begin{tabular}{l p{.7cm} p{.7cm} p{.7cm} p{.7cm} p{.7cm} p{.7cm} p{.7cm} p{.7cm}  p{.7cm}}
Elements of $\Z/16\Z${\Large${\mathstrut}$}&\hfill$0$\hfill&\hfill$1$\hfill&\hfill$2$\hfill&\hfill$3$\hfill&
\hfill$4$\hfill&\hfill$5$\hfill&\hfill$6$\hfill&\hfill$7$\hfill&\hfill$8$\hfill\\\hline
Points on $\Cabcd${\Large${\mathstrut}$}&\hfill$\nO$\hfill&\hfill$-\konj{B}$\hfill& 
\hfill$A$\hfill&\hfill$\ $\hfill&\hfill$D$\hfill&\hfill$\ $\hfill&\hfill$-\konj{A}$\hfill
&\hfill$B$\hfill&\hfill$T$\hfill\\[3ex]
Points on $\Cabcd${\Large${\mathstrut}$}
&\hfill$\nO$\hfill&\hfill$\konj{B}$\hfill& 
\hfill$-A$\hfill&\hfill$\ $\hfill&\hfill$\konj{D}$\hfill&\hfill$\ $\hfill&\hfill$\konj{A}$\hfill&
\hfill$-B$\hfill&\hfill$T$\hfill\\\hline
Elements of $\Z/16\Z${\Large${\mathstrut}$}
&\hfill$16$\hfill&\hfill$15$\hfill&\hfill$14$\hfill&\hfill$13$\hfill&
\hfill$12$\hfill&\hfill$11$\hfill&\hfill$10$\hfill&\hfill$9$\hfill&\hfill$8$\hfill
\end{tabular}
\end{center}
Because $A$ and $\konj{A}$ are on $\Cabcd$, we have $\beta=1$ and $\delta=-(\alpha+\gamma)$.
Now, since $\dritt A{\konj{A}}=D$, the point $D$ is on the line passing through $A$ and $\konj{A}$,
and therefore, $D=(x_0,x_0)$ for some $x_0\in\F$. Since $\dritt DD=T$, this implies that the tangent
to the curve $\Cabcd$ with contact point~$D$ is parallel to the $x$-axis, which gives us
\begin{eqnarray}\label{eq:x0}
x_0=-(\alpha+\gamma)\pm\sqrt{(\alpha+\gamma)^2-1}\,.
\end{eqnarray}
Furthermore, we have $-D=\konj{D}$, which implies that 
for $\konj{D}=(\konj{x},-x_0)$ we have $x_0=\konj{x}_0$,
where by {\FCT}\;\ref{fct:konj},
$$\konj{x}_0=\frac{\alpha+\delta(x_0\gamma+\beta)}{\gamma(x_0-\delta)}.$$
Thus, by~(\ref{eq:x0}) we obtain $$\gamma=\frac{1-\alpha^2}{2\alpha}\,,$$
and for a point $P=(x,y)$ on $\Cabcd$, the $x$-coordinate of $\konj{P}$ is
$$\konj{x}=-\frac{x+2\alpha+x\alpha^2}{1+2x\alpha+\alpha^2}.$$

Now, let us consider the points $B$ and $\konj{B}$.
Since $\dritt B\konj{B}=\konj{A}$, the point $\konj{A}$ is on the line $g$ passing 
through $B$ and $\konj{B}$. Let $\lambda$ be the slope of $g$, 
then $$g(x)=\lambda x+(\lambda-1)\,.$$ Because the line $g$ passes through the 
the points $B=(x_1,y_1)$ and $\konj{B}=(\konj{x}_1,-y_1)$, we must 
have $g(x_1)=-g(\konj{x}_1)$, and solving this equation for $\lambda$ gives
$$\lambda_0 =\frac{1+2x_1\alpha+\alpha^2}{1+2x_1\alpha+\alpha^2+\alpha(x_1^2-1)}\,.$$
Furthermore, we must have that the points $B$ and $\konj{B}$ are on $\Cabcd$, {\ie},
$$g_{\lambda_0}(x_1)^2=\bigl(\lambda_0 x_1+(\lambda_0-1)\bigr)^2 = 
\frac{\alpha+\beta x_1+\gamma x_1^2}{x_1-\delta}=
\frac{x^2 + 2 x \alpha - \alpha^2(x^2-2)}{1 + 
 2 x \alpha + \alpha^2}\,,$$
which finally gives us the following four values for $x_1$:
$$
-1-\frac{\sqrt{1-\alpha^{\mathstrut 2}}}{1+\alpha}\pm
\frac{\sqrt{\alpha\,(\alpha^{\mathstrut 2}-1)\Bigl(1+\sqrt{1-\alpha^{\mathstrut 2}}-
\alpha\bigl(1 + \alpha +\sqrt{1-\alpha^{\mathstrut 2}}\,\bigr)\Bigr)}}{\alpha(1+\alpha)}
$$
$$
-1+\frac{\sqrt{1-\alpha^{\mathstrut 2}}}{1+\alpha}\pm
\frac{\sqrt{\alpha\,(\alpha^{\mathstrut 2}-1)\Bigl(1-\sqrt{1-\alpha^{\mathstrut 2}}-
\alpha\bigl(1 + \alpha -\sqrt{1-\alpha^{\mathstrut 2}}\,\bigr)\Bigr)}}{\alpha(1+\alpha)}
$$
Since $x_1$ and $\alpha$ belong to $\F$, this implies that at least one of
$$z_{1,3}=\alpha\sqrt{1-\alpha^{\mathstrut 2}}\pm\sqrt{\alpha\,(\alpha^2-1)\Bigl(1+\sqrt{1-\alpha^2}-
\alpha\bigl(1 + \alpha +\sqrt{1-\alpha^2}\,\bigr)\Bigr)}$$
and 
$$z_{2,4}=\alpha\sqrt{1-\alpha^{\mathstrut 2}}\pm\sqrt{\alpha\,(\alpha^2-1)\Bigl(1-\sqrt{1-\alpha^2}-
\alpha\bigl(1 + \alpha-\sqrt{1 - \alpha^2}\,\bigr)\Bigr)}$$
belongs to~$\F$. On the other hand, if at least one of $z_1,z_2,z_3,z_4$ belongs to~$\F$, 
then also the corresponding~$x_1$ belongs to~$\F$, which shows that there exists an elliptic curve 
$\Cab$ with torsion group $\Z/16\Z$ over~$\F$.
\end{proof}

As a consequence of {\THM}\;\ref{thm:mainZ/16} we obtain
\begin{cor}
Let\/ $m\in\Q\setminus\{-1,0,1\}$ and let
\begin{align*}
d_1 &=(m^4 - 1)(m^2 - 2 m - 1)\\[.8ex]
d_2 &= m (m^2 + 1) (m^2 + 2 m - 1) 
\end{align*}
Then for each\/ $i\in\{1,2\}$ there is an elliptic curve 
over\/ $\Q(\sqrt{d_i\mathstrut})$ with torsion group~$\Z/16\Z$.
\end{cor}

\begin{proof} If $\alpha\in\Q$ is such that $\sqrt{1-\alpha^{\mathstrut 2}}\in\Q$,
then $\alpha=\alpha_1$ or $\alpha=\alpha_2$, where
$$\alpha_1=\frac{m^2 - 1}{m^2 + 1}\qquad\text{or}\qquad\alpha_2=\frac{2m}{m^2 + 1}\,.$$
Let $z_1,z_2$ be as above. If we substitute $\alpha$ by $\alpha_1$ into $z_1$, then
$z_1\in\Q(\sqrt{d_1})$, and if we substitute $\alpha$ by $\alpha_2$ into $z_2$, then
$z_2\in\Q(\sqrt{d_2})$. 
\end{proof}

For each~$m_1\neq \pm 1$, let 
$m_2:=1/m_1$, $m_3:=(m_1-1)/(m_1+1)$, $m_4:=1/m_3$, 
and for $j\in\{1,2,3,4\}$ let \hbox{$m_{-j}:=-m_j$.} 
Furthermore, for $i\in\{1,2\}$ and $k\in\{\pm 1,\pm 2,
\pm 3,\pm 4\}$ let $d_{i,k}$ be the value of $d_i$ obtained from $m_k$.
Now, for each value~$m_1\neq \pm 1$ we obtain four groups of four pairwise isomorphic 
elliptic curves with torsion group $\Z/16\Z$ over the same quadratic 
field $\Q(\sqrt{d\mathstrut})$. The four groups are
given by the following pairs~$(m_k,d_{i,k})$:
	$$\begin{array}{lllll}
		\text{over quadratic field}\hfill\Q(\sqrt{d_{1,1}\mathstrut})\hfill{:}
		&(m_1,\,d_{1,1})&(m_{-2},\,d_{1,-2})&(m_3,\,d_{2,3})&(m_{-4},\,d_{2,-4})\\[.8ex]
		\text{over quadratic field}\hfill\Q(\sqrt{d_{2,1}\mathstrut})\hfill{:}
		&(m_1,\,d_{2,1})&(m_{-2},\,d_{2,-2})&(m_3,\,d_{1,3})&(m_{-4},\,d_{1,-4})\\[.8ex]
		\text{over quadratic field}\ \Q(\sqrt{d_{1,-1}\mathstrut}):
		&(m_{-1},\,d_{1,-1})&(m_{2},\,d_{1,2})&(m_{-3},\,d_{2,-3})&(m_{4},\,d_{2,4})\\[.8ex]
		\text{over quadratic field}\ \Q(\sqrt{d_{2,-1}\mathstrut}):
		&(m_{-1},\,d_{2,-1})&(m_{2},\,d_{2,2})&(m_{-3},\,d_{1,-3})&(m_{4},\,d_{1,4})
	\end{array}$$
%

\noindent For example, over $\Q(\sqrt{d_1\mathstrut})$ we obtain the parametrisation
\begin{equation}\label{Z16d1}
y^2=x^3+\left((m^4 - 1)^2 - 4 m^2 (m^4 + 1)\right)x^2+\,16 m^8\,x\,,
\end{equation}
or equivalently
\begin{equation}
y^2+\left(\frac{m^4+2m^2-1}{m^2}\right)xy+(m^4-1)\,y=x^3+(m^2 - 1) \,x^2\,,
\end{equation}
and over $\Q(\sqrt{d_2\mathstrut})$, with respect to another normal form,
we obtain the parametrisation
\begin{equation}
y^2 + (1-c)\,xy -b\,y=x^3-b\,x^2$$
where $$b=-\frac{m(m-1)^2}{(m^2+1)^2}\qquad\text{and}\qquad 
c=-\frac{2m(m-1)^2}{(m^2+1)(m+1)^2}.
\end{equation}

\subsection{High Rank Elliptic Curves with Torsion Group $\boldsymbol{\Z/16\Z}$}
In~\cite[p.\,38]{Rabarison}, Rabarison listed a single $\Z/16\Z$ curve of 
rank~$1$ over $\Q(\sqrt{10})$, and in~\cite{ADJBP}, an example of a curve of rank~$2$ over 
$\Q(\sqrt{1\,785})$ is provided. In~\cite{Na}, 
Najman used the $2$-isogeny method to construct a $\Z/16\Z$ curve over $\Q(\sqrt{34\,720\,105})$ 
of proven rank~$3$ and conditional rank~$4$, starting from a rank-$3$ curve with the 
torsion group $\Z/2\Z\times\Z/8\Z$. The three curves mentioned above can be 
reproduced by using formula\;(\ref{Z16d1}) for $m=3$, $m=4$, and $m=12/17$, respectively.

As~\cite{DujellaQuad} lists only the smallest, by magnitude, 
$d$-values for quadratic fields $\Q(\sqrt{d})$, the third author has shown that a 
$\Z/16\Z$ curve of conditional rank~$4$ can be built over 
$\Q(\sqrt{17\,381\,446})$ by using $m=29/65$ in formula\;(\ref{Z16d1}). This 
is a current record for $\Z/16\Z$ curves of conditional rank~$4$.

By implementing the $2$-isogeny method~\cite{Na}, we have 
found new $\Z/16\Z$ curves of rank~$3$ for $m=5/8$, $m=3/13$, and $m=13/3$ over 
$\Q(\sqrt{413\,049})$, $\Q(\sqrt{105\,910})$, and $\Q(\sqrt{36\,490})$, 
respectively. For $m=3/13$, it required a point search performed by {\tt MAGMA} 
up to height $10^{16}$ on the $8$-coverings of the quadratic 
twist by~$d=105\,910$.  

By using $m=-9$ in formula\;(\ref{Z16d1}), we found a new $\Z/16\Z$ curve over 
$\Q(\sqrt{205})$ isomorphic to 
$$y^2 = x^3 + 10226878\,x^2 + 43046721\,x\,,$$
that ties a current record for $\Z/16\Z$ curves of rank~$3$, 
with the three generators
\begin{align*}
	P_1&= (4961, -1108800\,\sqrt{205}),\\
	P_2&= (67081/81, -1935672760/729),\\
	P_3&= (279524961, -332320219200\,\sqrt{205}).
\end{align*}

Uncovering generators on the $d$-twists to prove rank~$4$ unconditionally for $\Z/16\Z$ 
curves over $\Q(\sqrt{d})$ remains a challenge. The only {\tt MAGMA} calculation that 
has not resulted in a crash corresponds to the curve with $m=12/17$. 
After successfully performing both $8$-descent and $3$-descent in $68$~hours, 
no generator was found on the quadratic twist.
\bigskip

\noindent {\bf Acknowledgement.} We would like to thank Andrej Dujella for 
many fruitful hints. We appreciate Zev Klagsbrun's expertise 
that resolved isogeneity challenges for $\Z/14\Z$ curves, and his efforts to 
uncover the last generator in the mentioned $\Z/10\Z$ case. 
We commend Yusuf AttarBashi for sieving promising candidates in $\Z/10\Z$ 
and $\Z/12\Z$ torsion groups, and discovering many high-rank 
exemplars in the process.


\vspace{-2ex}
\bibliographystyle{plain}

\end{document}